\documentclass[leqno,11pt,draft]{amsart}

\usepackage{amssymb}
\usepackage{amsmath}
\usepackage{enumerate}
\usepackage{amsfonts}
\usepackage{color}
\usepackage[final]{graphicx}

\newtheorem{teo}[equation]{Theorem}
\newtheorem{coro}[equation]{Corollary}
\newtheorem{lema}[equation]{Lemma}
\newtheorem{propo}[equation]{Proposition}

\numberwithin{equation}{section}

\newcommand{\w}{\omega}

\newcommand{\e}{\varepsilon}

\newcommand{\VV}{\mathbb{V}}
\newcommand{\WW}{\mathbb{W}}
\newcommand{\RR}{\mathbb{R}}

\newcommand{\NN}{\mathbb{N}}

\newcommand{\QQ}{\mathbb{Q}}

\renewcommand{\ss}{\mathcal{S}}

\newcommand{\rr}{\mathcal{R}}

\newcommand{\Rn}{{\RR^n}}

\newcommand{\dxi}{\partial_i}

\newcommand{\iie}{\int_\e^{\e^{-1}}}

\newcommand{\wt}{\widetilde}

\newcommand{\8}{\infty}

\renewcommand{\(}{\left(}
\renewcommand{\)}{\right)}

\newcommand{\SumNN}{\sum_{k=1}^\infty}
\newcommand{\Summ}{\sum_{j =1}^d}

\newcommand{\supp}{\mathrm{supp}}

\newcommand{\Har}{H_L^1(\RR^n)}
\newcommand{\VVV}{\VV_1}
\newcommand{\VVVV}{\VV_1^\bot}
\newcommand{\HarD}{H_\Delta^1(\RR^n)}
\newcommand{\LL}{L^1(\RR^n)}

\newcommand{\dtt}{\frac{dt}{\sqrt{t}}}
\newcommand{\qq}{\mathcal{Q}}
\newcommand{\ww}{\mathcal{W}}
\newcommand{\ie}{{i,\e}}

\begin{document}

\title[Hardy spaces related to Schr\"odinger operators]{Hardy spaces related to Schr\"odinger operators with potentials which are sums of $L^p$-functions }
\author[Jacek Dziuba\'nski ]{ Jacek Dziuba\'nski }
\author[Marcin Preisner]{ Marcin Preisner}
\thanks{The research was partially supported by Polish Government funds for science - grants N N201 397137 and N N201 412639, MNiSW}
\subjclass[2000]{42B30, 35J10 (primary), 42B35, 42B20 (secondary)}

\begin{abstract}
We investigate the Hardy space $H^1_L$ associated to the  Schr\"odinger operator $L=-\Delta +V$ on $\Rn$, where $V=\sum_{j=1}^d V_j$. We assume that each $V_j$ depends on variables from a linear subspace $\VV_j$ of $\Rn$, $\dim \VV_j \geq 3$, and $V_j$ belongs to $L^q(\VV_j)$ for certain $q$. We prove that there exist two distinct isomorphisms of $H^1_L$ with the classical Hardy space. As a corollary we deduce a specific atomic characterization of $H_L^1$. We also prove that the space $H_L^1$ is described by means of the Riesz transforms $\mathcal R_{L,i} = {\partial_i}L^{-1/2}$.
\end{abstract}

\address{Instytut Matematyczny\\
Uniwersytet Wroc\l awski\\
Pl. Grunwal\-dzki 2/4\\
50-384 Wroc\l aw \\
Poland} \email{jdziuban@math.uni.wroc.pl,
preisner@math.uni.wroc.pl (corresponding author)}

\maketitle

\begin{flushleft}
{\small
{\it Keywords:} Schr\"odinger operator, Hardy space, maximal function, atomic decomposition, Riesz transform.}
\end{flushleft}
\section{Introduction and main results}
In the paper we consider a Schr\"odinger operator on $\RR^{n}$ given by
$$Lf(x)=-\Delta f(x) +V(x)f(x),$$
where $\Delta$ denotes the Laplace operator.
During the whole paper we assume that the potential $V$ satisfies:
\begin{enumerate}
\item[$(A_1)$] there exist $V_j\geq 0$, $V_j \not\equiv 0$ such that
$$V(x) = \Summ V_j(x),$$
\item[$(A_2)$] for every $j\in \{1,...,d\}$ there exists a linear subspace $\VV_j $ of $ \Rn$ of dimension $n_j\geq 3$ such that if $\Pi_{\VV_j}$ denotes the orthogonal projection on $\VV_j$ then
$$V_j(x) = V_j(\Pi_{\VV_j}x),$$
\item[$(A_3)$] there exists $\kappa>0$ such that for $j=1,...,d$ we have
$$V_j \in L^r(\VV_j)$$
for all $r$ satisfying $|r-n_j/2|\leq \kappa$.
\end{enumerate}

Denote by $K_t = \exp(-tL)$ and $P_t = \exp(t\Delta)$ the  semigroups of linear operators associated with $L$ and $\Delta$ respectively. Let $K_t(x,y)$ and $P_t(x-y)$ denote the integral kernels of these semigroups. The Feynman-Kac formula implies that

\begin{equation}\label{FKac}
0\leq K_t(x,y) \leq P_t(x-y)=(4\pi t)^{-n/2} \exp\(-|x-y|^2/4t\).
\end{equation}

Let $M_L$ and $M_\Delta$ be the associated maximal operators, i.e.,
$$M_Lf(x) = \sup_{t>0} |K_t f(x)| \qquad M_\Delta f(x) = \sup_{t>0} |P_t f(x)|.$$
The Hardy spaces $\Har$ and $\HarD$ are the subspaces of $\LL$ defined by:
$$f \in \Har \iff M_Lf \in \LL, \qquad f \in \HarD \iff M_\Delta f \in \LL $$
with the norms:
$$ \| f\|_{\Har}=\| M_Lf\|_{\LL}, \qquad \| f\|_{\HarD}=\| M_\Delta f\|_{\LL}.$$
Clearly the space $\HarD$ is the classical Hardy space $H^1(\Rn)$ (see \cite{Stein}).
The goal of the paper is to prove some characterizations of the space $\Har$.

Denote by $L^{-1}$ and $(-\Delta)^{-1}$ the operators with the kernels $\Gamma (x,y) =\int_0^\8 K_t(x,y) \, dt$ and $\Gamma_0 (x-y)=\int_0^\8 P_t(x-y) \, dt$.
 Clearly,
 \begin{equation}\label{kiu}
0\leq \int_0^t K_s(z,y) \, ds \leq \Gamma (z,y) \leq \Gamma_0(z-y)=  C |z-y|^{2-n}.
\end{equation}
We shall see that operators $I-VL^{-1}$ and $I-V\Delta^{-1}$ are bounded on $\LL$ and give the following characterization of the Hardy space $\Har$.

\begin{teo}\label{tw1}
Assume $f\in L^1(\RR^n)$. Then $f$ belongs to $\Har$ if and only if $(I-VL^{-1})f$ belongs to the classical Hardy space $H_\Delta^1(\RR^n)$.
Moreover,
$$\|f\|_{H^1_L(\RR^n)} \sim \|(I-VL^{-1})f\|_{H^1_\Delta(\RR^n)}.$$
\end{teo}

We define the auxiliary function $\w$ by
$$\w(x) = \lim_{t\to \8} \int_\Rn K_t(x,y)\, dy.$$
The above limit exists because, by \eqref{FKac} and the semigroup property, the function $t \mapsto K_t \mathbf{1}(x)$ is decreasing and takes values in $[0,1]$. Clearly, for every $t>0$,
\begin{equation}\label{omega22}
 \w(x)=K_t \w(x)=\int_{\mathbb R^n} K_t(x,y)\w(y)\, dy.
\end{equation}
We shall prove that there exists $\delta >0$ such that $\delta\leq \w (x)\leq 1$ (see Proposition \ref{lem_w}). We are  now in a position to state our second main result.

\begin{teo}\label{tw2}
 Let $f\in L^1(\RR ^n)$. Then $f$ belongs to $\Har$ if and only if $\w f$ belongs to $\HarD$. Additionally,
$$\|f\|_{\Har} \sim \|\w f\|_{\HarD}.$$
\end{teo}

From Theorem \ref{tw2} we get atomic characterizations of the elements of $\Har$. We call a function $a$ an $\w$-atom if it satisfies:
\begin{itemize}
\item there exists a ball $B=B(y,r)$ such that $\supp \, a \subseteq B$,
\item $\|a\|_\8 \leq |B|^{-1}$,
\item $\int_\Rn a(x) \w (x) \, dx = 0$.
\end{itemize}

\begin{coro}
If a function $f$ belongs to $\Har$ then there exist a sequence $a_k$ of  $\w$-atoms and a sequence $\lambda_k\in \mathbb C$ such that $\sum_{k=1}^\8 |\lambda_k| <\8$, $f = \sum_{k=1}^\8 \lambda_k a_k$, and
$$\|f\|_{\Har} \sim \sum_{k=1}^\8 |\lambda_k|.$$
\end{coro}

For $i=1,...,n$ denote by $\partial_i$ the derivative in the direction of the $i$-th canonical coordinate of  $\Rn$. For $f\in \LL$ the classical Riesz transforms $\rr_{\Delta, i}$ are given by
$$\rr_{\Delta,i}f = \lim_{\e\to 0} \iie \dxi P_t f \dtt.$$
Similarly we define the Riesz transforms $\rr_{L,i}$ associated with $L$ by setting
$$\rr_{L,i}f = \lim_{\e\to 0} \iie \dxi K_t f \dtt.$$
We shall see that the last limits are  well-defined in the sense of distributions and they characterize $\Har$ in the following sense.

\begin{teo}\label{tw3}
An $L^1(\RR ^n)$-function $f$ belongs to $\Har$ if and only if $\rr_{L,i} f$ belong to $\LL$ for $i=1,...,n$. Additionally,
$$\|f\|_{\Har} \sim \| f\|_{\LL}+\sum_{i=1}^n \|\rr_{L,i} f\|_{\LL}.$$
\end{teo}

Hardy spaces associated with semigroups of linear operators and in particular Schr\"odinger semigroups  attracted attention of many authors,  see, e.g., \cite{ADM}, \cite{BZ}, \cite{CZ}, \cite{DZ1}, \cite{DZ5}, \cite{Hof}  and references therein.
The present paper generalizes the results of   \cite{DZ} and \cite{DP}, where the spaces    $\Har$ were  studied
under assumptions: $V\geq 0$, $\supp \, V$ is compact, $V\in L^r(\RR^n)$ for some $r>n/2$. Obviously such potentials $V$ satisfy the conditions  $(A_1)-(A_3)$.
To prove Theorems \ref{tw1}, \ref{tw2}, and \ref{tw3} we develop methods  of \cite{DZ} and \cite{DP}.

\section{Auxiliary lemmas}
In the paper we shall use the following notation. For $z\in \Rn$ and a subspace $\VV_j$ of $\Rn$ we write
$$z = z_j + \wt z_j, \quad z_j = \Pi_{\VV_j}(z) , \quad \wt z_j = \Pi_{\VV_j^\bot}(z),\quad \wt n_j = \dim \VV_j^\bot = n-n_j.$$
Notice that if $\VV_j=\Rn$, then , in fact, there is no $\VV_j^\bot$ in fact.

The relation between $P_t$ and $K_t$ is given by the perturbation formula.
\begin{equation}\label{pert}
P_t = K_t + \int_0^t P_{t-s}VK_s\, ds.
\end{equation}
The following two lemmas state crucial estimates that will be used in many proofs of this paper.

\begin{lema}\label{Lr}
There exists $\lambda>0$ such that
\begin{equation}\label{Lrrr}
\sup_{y\in \Rn} \|V(\cdot)|\cdot-y|^{2-n+\mu}\|_{L^r(\RR^n)} \leq C
\ \ \ \text{for }  r\in [1,1+\lambda] \  \text{and  } \mu \in [-\lambda,\lambda].
\end{equation}
\end{lema}
\begin{proof}
 It suffices to prove (\ref{Lrrr}) for $V=V_1$. For fixed  $y\in \RR^n$ we have
\begin{equation}\label{prew}
\| V_1(\cdot )|\cdot -y|^{2-n+\mu} \|_{L^r(\RR^n)}^r \leq C \int_{\VVV} \int_{\VVVV} \frac{V_1(z_1)^r}{|z_1-y_1|^{-r(2-n+\mu)}+|\wt z_1 - \wt y_1|^{-r(2-n+\mu)}}\, d\wt z_1 \, dz_1.
\end{equation}
Observe that if $\lambda >0$ is sufficiently small, $ r\in [1,1+\lambda]$, and $\mu \in [-\lambda,\lambda]$ then
\begin{equation}\label{edre}
\begin{split}
\int_{\VVVV}&\(|z_1-y_1|^{-r(2-n+\mu)}+|\wt z_1 - \wt y_1|^{-r(2-n+\mu)}\)^{-1} \, d\wt{z_1}\cr
&\leq C \int_{|z_1-y_1|>|\wt z_1 - \wt y_1|} |z_1-y_1|^{r(2-n+\mu)}\, d\wt z_1 +C\int_{|z_1-y_1|\leq|\wt z_1 - \wt y_1|} |\wt z_1 - \wt y_1|^{r(2-n+\mu)} \, d\wt z_1\cr
&\leq C |z_1-y_1|^{r(2-n+\mu)+\wt n_1}.
\end{split}
\end{equation}
Thus, by (\ref{edre}),
\begin{align}
\| V_1(\cdot)|\cdot -y|^{2-n+\mu} \|_{L^r(\RR^n)}^r \leq &C \int_{|z_1-y_1|\leq 1} V_1(z_1)^r |z_1-y_1|^{r(2-n+\mu)+\wt n_1} \, dz_1 \cr
&+C \int_{|z_1-y_1|>1} V_1(z_1)^r |z_1-y_1|^{r(2-n+\mu)+\wt n_1} \, dz_1.
\end{align}
Note that by $(A_3)$ there exist $t,s>1$ such that $V_1^r \in L^t(\VVV) \cap L^s(\VVV)$ and
$$\chi_{\{|z_1|\leq 1\}}(z_1)|z_1|^{r(2-n+\mu)+\wt n_1} \in L^{t'}(\VVV), \qquad  \chi_{\{|z_1|> 1\}}(z_1)|z_1|^{r(2-n+\mu)+ \wt n_1} \in L^{s'}(\VVV)$$
for $r\in[1,1+\lambda]$ and $\mu \in [-\lambda,\lambda]$ provided $\lambda>0$ is small enough. Thus \eqref{Lrrr} follows from  the H\"older inequality.
\end{proof}

\begin{coro}
The operators $I-V\Delta^{-1}$ and $I-VL^{-1}$ are bounded on $\LL$ and
\begin{equation}\label{odwrocenie}
(I-VL^{-1})(I-V\Delta^{-1}) f  =(I-V\Delta^{-1})(I-VL^{-1}) f=f \quad \text{for} \quad f\in \LL.
\end{equation}
\end{coro}

\begin{lema}\label{Lrr}
There exists $\sigma, \e>0$ such that for $s\in [1,1+\e]$ and $R\geq 1$ we have
\begin{equation}\label{eq1} \sup_{y\in \Rn} \int_{|z-y|>R} V(z)^s|z-y|^{s(2-n)}\, dz \leq C R^{-\sigma}.
\end{equation}
\end{lema}
\begin{proof}
It is enough to prove (\ref{eq1}) for $V=V_1$. Fix $q>1$ and $\e >0$ such that   $n_1/q(1+\e)-2>0$ and $V_1\in L^{q(1+\e)}(\VVV)\cap L^q(\VVV)$  (see (A3)). Set $\sigma = n_1/q -2$. For $s\in [1,1+\e]$ we have
\begin{align}
\int_{|z-y|>R} V_1(z)^s|z-y|^{s(2-n)}\, dz \leq &\int_{|z_1-y_1|\geq |\wt z_1 - \wt y_1|} \chi_{\{|z-y|>R\}}(z) V_1(z)^s|z_1-y_1|^{s(2-n)}\, dz\cr
 &+ \int_{|z_1-y_1|< |\wt z_1 - \wt y_1|}\chi_{|z-y|>R}(z) V_1(z)^s|\wt z_1-\wt y_1|^{s(2-n)}\, dz\cr
= &T(R)+S(R).
\end{align}
 If $|z_1-y_1|\geq |\wt z_1 - \wt y_1|$ and $|z-y|>R\geq 1$, then $|z_1-y_1|>R/2\geq 1\slash 2$. Thus,
\begin{align}
T(R)&\leq C \int_{|z_1-y_1|>R/2} |z_1-y_1|^{n- n_1}V_1(z_1)^s |z_1-y_1|^{s(2-n)}\, dz_1 \cr
&\leq C \|V_1\|^s_{L^{qs}(\RR^{n_1})} \(\int_{|z_1-y_1|>R/2} |z_1-y_1|^{(s(2-n)+n-n_1)q'}\, dz_1\)^{1/q'}= CR^{-\sigma}.
\end{align}
Similarly, if $|z_1-y_1|< |\wt z_1 - \wt y_1|$ and $|z-y|>R\geq 1$, then $|\wt z_1-\wt y_1|>R/2\geq 1\slash 2$ and
\begin{align}
S(R)&\leq C \int_{|\wt z_1-\wt y_1|>R/2} \|V_1\|^s_{L^{sq}(\RR^{n_1})} \(\int_{|z_1-y_1|< |\wt z_1 - \wt y_1|}dz_1\)^{1/q'} |\wt z_1-\wt y_1|^{s(2-n)}\, d\wt z_1 \cr &\leq C \int_{|\wt z_1-\wt y_1|>R/2} |\wt z_1-\wt y_1|^{s(2-n)+n_1/q'}\, d\wt z_1= CR^{-\sigma}.
\end{align}
\end{proof}

We shall need the following properties of the function $\w$ similar to those that hold in the case of compactly supported potentials  (c.f., \cite{DZ}, Lemma 2.4).
\begin{propo}\label{lem_w}
There exist $\gamma, \delta >0$ such that for $x,y \in \Rn$ we have
\begin{enumerate}[(a)]
\item \label{prop_a} $|\w(x)-\w(y)| \leq C_\gamma|x-y|^\gamma$,
\item\label{prop_b} $\delta\leq \w(x) \leq 1$.
\end{enumerate}
\end{propo}
\begin{proof}
 The property $\eqref{prop_a}$ can be proved by a slight modification of the proof of   $(2.6)$ in \cite{DZ}. Indeed, thanks to (\ref{omega22}) and $0\leq \w(x)\leq 1$, it suffices to show that there is $C,\gamma>0$ such that for $|h|<1$ we have
 \begin{equation}\label{L1} \int_{\mathbb R^n} |K_1(x+h,y)-K_1(x,y)|\, dy\leq C|h|^\gamma.
 \end{equation}
To this purpose, by using  (\ref{pert}), it is enough to establish that
 $$ \sum_{j=1}^d \int_{\mathbb R^n} \Big| \int_{0}^1 \int_{\mathbb R^n} (P_s(x+h-z)-P_s(x-z))V_j(z)K_{1-s}(z,y)\, dz\, ds\Big|\, dy\leq C|h|^\gamma.$$
 Consider one summand that contains $V_1$. Utilizing the fact that $P_s(x)=P_s(x_1)P_s(\tilde x_1)$, where $P_s(x_1)$ and $P_s(\tilde x_1)$ are the heat kernels on \
 $\VVV$ and $\VVV^\perp$ respectively,  we have
 \begin{equation}\begin{split}
  I & =\int_{\mathbb R^n}   \Big| \int_{0}^1 \int_{\mathbb R^n} (P_s(x+h-z)-P_s(x-z))V_1(z)K_{1-s}(z,y)\, dz\, ds\Big|\, dy \\
  &\leq    \int_{0}^1 \int_{\mathbb R^n} | P_s(x+h-z)-P_s(x-z)|V_1(z)\, dz\, ds\\
  &\leq \int_{0}^1 \int_{\mathbb R^n} P_s(x_1+h_1-z_1)\big|P_s(\tilde x_1+\tilde h_1-\tilde z_1)-P_s(\tilde x_1-\tilde z_1)\big|V_1(z_1)\, dz\, ds\\
  & \ \ + \int_{0}^1 \int_{\mathbb R^n} P_s(\tilde x_1-\tilde z_1)\big|P_s(x_1+ h_1- z_1) -P_s( x_1- z_1)\big|V_1(z_1)\, dz\, ds\\
   \end{split}\end{equation}
By taking $q>n_1\slash 2$ such that $V_1\in L^q(\VVV)$ and using the H\"older inequality we obtain
  \begin{equation}\begin{split}
I &\leq \int_{0}^1  \|P_s(x_1)\|_{L^{q'}(dx_1)} \| V_1(z_1)\|_{L^q(dz_1)}
 \int_{\VVV^\perp} |P_s(\tilde x_1+\tilde h_1-\tilde z_1)
  -P_s(\tilde x_1-\tilde z_1)|\, d\tilde z_1\, ds\\
  & \ \ + \int_{0}^1 \( \int_{ \VVV} \big|P_s(x_1+ h_1- z_1)
  -P_s(x_1- z_1)\big|^{q'}\, dz_1\)^{1\slash q'}\| V_1(z_1)\|_{L^q(dz_1)}\, ds\\
  & \leq C(|\tilde h_1|^\gamma +|h_1|^\gamma ),
  \end{split}
 \end{equation}
which finishes the proof of (a).

 Next we note that
 \begin{equation}\label{ooo}
 K_t(x,y)>0 \ \ \text{ for } t>0 \  \text{and } x,y\in \mathbb R^n.
 \end{equation}
 The proof of (\ref{ooo}) is a straightforward  adaptation of the proof of \cite[Lemma 2.12]{DZ}. We omit the details.

Our next task is to establish that there exists $\delta>0$ such that
\begin{equation}\label{oddz}
 \w(x) \geq \delta.
\end{equation}
The proof of (\ref{oddz}) goes by induction on $d$. Assume first that we have only one potential $V_1$, that is, $d=1$. Then, $K_t(x,y)=K^{\{1\}}_t(x_1,y_1)P_t(\tilde x_1-\tilde y_1)$, where $K^{\{1\}}_t(x_1,y_1)$ is the kernel of the semigroup generated by $\Delta -V_1(x_1)$   on $\VVV$ and $P_t(\tilde x_1)$ is the classical heat semigroup on $\VVVV$. Hence $\w(x)=\w_0(x_1)$, where $\w_0(x_1)=\lim_{t\to\infty }\int_{\VVV} K^{\{1\}}_t(x_1,y_1)\, dy_1$. Therefore, there is no loss of generality in proving (\ref{oddz}) if we assume that $\VVV=\Rn$.
If we integrate \eqref{pert} over $\Rn$ and take the limit as $t\to \8$, then we get
\begin{equation}\label{eqeq}
1-\w(x) = \int_{\Rn} V(y)\Gamma (x,y) \, dy,
\ \ \text{where   } \Gamma(x,y)\leq C |x-y|^{2-n}.
\end{equation}
By $(A_3)$ and the H\"older inequality we can find $t,s>1$ such that $V\in L^t(\RR^{n})\cap L^s(\RR^{n})$, $\chi_{\{|x|\leq 1\}}(x)|x|^{2-n}\in L^{t'}(\RR^{n})$, and $\chi_{\{|x|>1\}}(x)|x|^{2-n}\in L^{s'}(\RR^{n})$. Thus \eqref{eqeq} leads to
\begin{equation}\label{klm}
\lim_{|x|\to \8} \int_\Rn V(y) |x-y|^{2-n} \, dy = 0 \quad \text{and} \quad \lim_{|x|\to \8}\w(x) =1.
\end{equation}
The equation (\ref{omega22}) combined with (\ref{ooo}) and (\ref{klm}) imply that $w(x)>0$ for every $x\in\Rn$. Since $\w$ is continuous (see (a)) and $\lim_{|x|\to\infty}\w(x)=1$, we get (\ref{oddz}).

Using induction, we assume that \eqref{oddz} is true for $V$ being a sum of $d-1$ potentials. Take $V=V_1+...+V_d$. As in the case of $d=1$, we can assume that $\text{lin}\{\VVV,..., \VV_d\} =\Rn$. Consider the  semigroup $\{S_t\}_{t>0}$ generated by $-\Delta +V_2+...+V_d$. Let $\w_1(x)=\lim_{t\to \8} \int_\Rn S_t(x,y)\, dy$. By the inductive assumption $\w_1(x)\geq \delta_1$. Similarly to \eqref{eqeq}, the perturbation formula
$$ S_t=K_t+\int_0^tS_{t-s}V_1K_s\, ds$$
implies
\begin{equation}\label{poiu}
\delta_1 \leq \w_1(y) \leq \w(y) + C \int_\Rn V_1(z)|z-y|^{2-n} \, dz \leq \w(y) + C\int_{\VVV} V_1(z_1)|z_1-y_1|^{2-n_1} \, dz_1,
\end{equation}
where the last inequality is proved in \eqref{edre}. If $y_1 \to \8$ then the integral on the right hand side  of \eqref{poiu} goes to zero. Hence, $\w(y)>\delta_1\slash 2$ provided $|y_1|>R_1$.
We repeat the argument for each $V_2$, ..., $V_d$ instead of $V_1$ and deduce that there exists $R,\delta>0$ such that $\w(x) >\delta$ for $|x|>R$. Consequently, by using (\ref{omega22}), (\ref{ooo}) and continuity of $\w$ we obtain (\ref{oddz}).
\end{proof}

\section{Proof of Theorem \ref{tw1}}

By \eqref{pert} we get
\begin{equation}\label{pert1}
K_t - P_t(I-VL^{-1}) = Q_t - W_t,
\end{equation}
where
\begin{equation*}\label{pert2}
W_t = \int_0^t (P_{t-s}-P_t)\,V\, K_s\, ds, \qquad Q_t = \int_t^\8 P_t \, V \, K_s \, ds.
\end{equation*}
\newcommand{\jjj}{{\langle j \rangle}}
\newcommand{\jjjj}{{\langle 1 \rangle}}
Let
\begin{align*}
&W_t(x,y) &&=\Summ W_t^{\jjj}(x,y)&&= \Summ \int_0^t \int_\Rn (P_{t-s}(x-z) - P_t (x-z)) V_j(z) K_s(z,y) \, dz\, ds, \\
&Q_t(x,y)&&=\Summ Q_t^{\jjj}(x,y)&&=\Summ \int_{\Rn}P_t(x,z)\int_t^\infty V_j(z) K_s(z,y)\, ds\, dz
\end{align*}
be the integral kernels of $W_t$ and $Q_t$ respectively.
In order to prove Theorem \ref{tw1} it is sufficient to establish that the maximal operators: $f\mapsto \sup_{t>0}|W_t f|$ and $f \mapsto \sup_{t>0}|Q_t f|$ are bounded on $L^1(\RR^n)$. The proofs of these facts are presented in the following four lemmas.

\begin{lema}\label{lem1}
The operator $f \mapsto \sup_{t>2} |W_t f|$ is bounded on $L^1(\RR^n)$.
\end{lema}
\begin{proof}
It suffices to prove that
$$\sup_{y\in \Rn}\int_\Rn \sup_{t>2} |W_t(x,y)| \, dx < \8.$$

Without loss of generality we can consider only $W_t^{\jjj}(x,y)$. For $0<\beta<1$, which will be fixed later on, we write
\begin{align*}
W_t^\jjjj(x,y) &= \int_0^t \int_\Rn (P_{t-s}(x-z) - P_t (x-z)) V_1(z) K_s(z,y) \, dz\, ds \cr
&= \int_0^{t^{\beta}}...+\int_{t^{\beta}}^t... = F_1(x,y;t) + F_2(x,y;t).
\end{align*}

To estimate $F_1$ observe that for $t>2$ and $s\leq t^{\beta}<t$ there exists $\phi \in \ss(\RR^n)$ such that
\begin{equation}\label{poi}
|P_{t-s}(x-z) - P_t (x-z)| \leq C \frac{s}{t} \phi_t(x-z).
\end{equation}
Here and subsequently $f_t(x) = t^{-n\slash 2} f(x/\sqrt{t})$ and $\ss$ denotes the Schwartz class of functions.
From \eqref{poi} and \eqref{kiu}, we get
$$|F_1(x,y;t)|\leq Ct^{-1+\beta} \int_\Rn \phi_t(x-z)V_1(z)|z-y|^{2-n} dz.$$
Since
$\sup_{t>2}t^{-1+\beta} \phi_t(x-z) \leq C (1+|x-z|)^{-n-2+2\beta}$,
we have that
$$\sup_{y\in \Rn} \int_\Rn \sup_{t>2} |F_1(x,y;t)|dx \leq C\int_\Rn V_1(z) |z-y|^{2-n} \, dz \leq C,$$
where the last inequality comes from Lemma \ref{Lr}.

To deal with $F_2$ we write
\begin{align*}
F_2(x,y;t)&= \int_{t^\beta}^t \int_\Rn P_{t-s}(x-z)V_1(z)K_s(z,y)\, dz\, ds - \int_{t^\beta}^t \int_\Rn P_{t}(x-z)V_1(z)K_s(z,y)\, dz\, ds\cr
 &= F'_2(x,y;t)-F''_2(x,y;t)
\end{align*}
Observe that for $s\in [t^{\beta},t]$ we have
\begin{equation}\label{eee}
K_s(z,y)\leq C t^{-\beta n/2} \exp\(-|z-y|^2/4t\).
\end{equation}
Also
\begin{equation}\label{rre}
\int_0^t P_{t-s}(x-z)\, ds=\int_0^t P_s(x-z)\, ds \leq C|x-z|^{2-n} \exp\(-|x-z|^2/ct\).
\end{equation}
As a consequence of \eqref{eee}--\eqref{rre} we obtain
\begin{align*}
F'_2(x,y;t)\leq C \int_\Rn t^{-\beta n/2} |x-z|^{2-n}\exp\(-|x-z|^2/ct\)V_1(z_1)\exp\(-|z-y|^2/4t\)\, dz.
\end{align*}
Then, for $\e >0$,
\begin{align*}
\sup_{t>2}t^{-\beta n/2} &\exp\(-|x-z|^2/ct\) \exp\(-|z-y|^2/4t\)\cr
&\leq C\sup_{t>2}t^{-1-\e} \exp\(-|x-z|^2/ct\)\cdot  \sup_{t>2}t^{-\beta n/2+1+\e} \exp\(-|z-y|^2/4t\)\cr
&\leq C(1+|x-z|)^{-2-2\e} |z-y|^{2+2\e-\beta n}.
\end{align*}
Consequently,
\begin{align*}
\sup_{y\in \Rn} \int_\Rn \sup_{t>2} F'_2(x,y;t) \, dx &\leq C \sup_{y\in \Rn} \int_\Rn \int_\Rn \frac{|x-z|^{2-n}}{(1+|x-z|)^{2+2\e}} |z-y|^{2+2\e-\beta n}V_1(z)\, dx\, dz\cr
&\leq C\sup_{y\in \Rn} \int_\Rn |z-y|^{2+2\e-\beta n}V_1(z_1)\, dz .
\end{align*}
If we choose $\beta<1$ close to 1 and $\e$ small, then we can apply Lemma \ref{Lr} and get
$$\sup_{y\in \Rn} \int_\Rn \sup_{t>2} F'_2(x,y;t) \, dx\leq C.$$

We now turn to estimate $ F''_2(x,y;t) $.
Observe that for $\e >0$ we have
$$\int_{t^\beta}^t K_s(z,y)\, ds \leq C\int_{t^\beta}^\8 t^{-\beta\e} s^{-n/2+\e} \exp\(-|z-y|^2/(4s)\)\, ds \leq C t^{-\beta \e} |z-y|^{2-n+2\e}.$$
Then from Lemma \ref{Lr} we conclude that
\begin{align*}
\sup_{y\in \Rn} \int_\Rn \sup_{t>2}F''_2(x,y;t) \, dx &\leq C\sup_{y \in \Rn} \int_\Rn \int_\Rn \sup_{t>2} t^{-\beta \e} P_t(x-z) V_1(z) |z-y|^{2-n+2\e}\, dx \, dz\cr
&\leq C\sup_{y\in \Rn}\int_\Rn \int_\Rn (1+|x-z|)^{-n-2\beta \e} V_1(z) |z-y|^{2-n+2\e}\, dx \, dz\cr
 &\leq C\sup_{y\in \Rn} \int_\Rn V_1(z) |z-y|^{2-n+2\e}\, dz\leq C,
\end{align*}
provided $\e>0$ is small enough.
\end{proof}

\begin{lema}\label{lem2}
The operator $f \mapsto \sup_{t\leq 2} |W_t f|$ is bounded on $L^1(\RR^n)$.
\end{lema}
\begin{proof}
It is enough to prove that
$$\sup_{y\in \Rn} \int_\Rn \sup_{t\leq 2} |W_t^\jjjj(x,y)| \, dx < \8.$$

We have
$$\int_0^t \int_\Rn (P_{t-s}(x-z) - P_t (x-z)) V_1(z) K_s(z,y) \, dz\, ds = \int_0^{t/2}...+\int_{t/2}^t... = F_3(x,y;t) + F_4(x,y;t).$$

To deal with $F_3$ observe that for $t\leq 2$, $s\leq t/2$ we have
$$|P_{t-s}(x-z) - P_t (x-z)| \leq C\phi_t(x-z),$$
where $\phi \in \ss (\RR^n)$, $\phi \geq 0$.
Therefore
$$\sup_{t\leq 2}|F_3(x,y;t)|\leq C \sup_{t\leq 2} \int_\Rn \phi_t(x-z)V_1(z)|z-y|^{2-n} dz.$$
Denote by $M_\phi^0$ the classical local maximal operator associated with $\phi$, that is,
$$M_\phi^0 f(x)=\sup_{t\leq 2} |\phi_t*f(x)|.$$
Then
$$\sup_{t\leq 2}|F_3(x,y;t)| \leq C M_\phi^0 (\xi_y)(x),$$
where $\xi_y(z)=V_1(z)|z-y|^{2-n}$. We claim that
\begin{equation}\label{rtyui}
\sup_{y\in \Rn}\int_\Rn \sup_{t\leq 2}|F_3(x,y)|dx \leq C\sup_{y\in \Rn} \int_{\Rn} M_\phi^0 (\xi_y)(x)\, dx \leq C.
\end{equation}
To obtain \eqref{rtyui} we write
$$\xi_y(z) = \sum_{k=1}^\8 \xi_{y,k}(z),$$
where$$\xi_{y,1}(z) = V_1(z)|z-y|^{2-n}\chi_{B(y,2)}(z), \quad \xi_{y,k}(z) = V_1(z)|z-y|^{2-n} \chi_{B(y,2^k)\backslash B(y,2^{k-1})}(z), \ \ k>1.$$
From Lemma \ref{Lr} it follows that there exists $s>1$ such that
\begin{equation}\label{llllll}
\supp \,\xi_{y,1} \subseteq B(y,2) \text{ and } \|\xi_{y,1}\|_{L^s(\Rn)} \leq C\leq C|B(y,2)|^{-1+1/s}.
\end{equation}

Consider $\xi_{y,k}$ for $k>1$. Set $q<n_1/2$ such that $V_1\in L^q(\VVV)$. Then
\begin{align}\label{lllll}
&& &\supp\, \xi_{y,k} \subseteq B(y, 2^k).\cr
&& &\|\xi_{y,k}\|_{L^{q}(\RR^n)} \leq C2^{k(2-n)} \|V_1\|_{L^{q}(\VVV)}2^{k(n-n_1)/q} \leq C |B(y, 2^k)|^{-1+1/q} 2^{-\rho k},
\end{align}
where $\rho=n_1/q -2$.
Now, our claim \eqref{rtyui} follows from \eqref{lllll}, \eqref{llllll}, and the classical theory of local maximal operators.

It remains to analyze $F_4=F_5-F_6$, where
\begin{align*}
F_5(x,y;t) &= \int_{t/2}^t \int_\Rn P_{t-s}(x-z) V_1(z) K_s(z,y) \, dz\, ds,\\
F_6(x,y;t) &= \int_{t/2}^t \int_\Rn P_t (x-z) V_1(z) K_s(z,y) \, dz\, ds.
\end{align*}
Clearly,
$$\sup_{s\in[t/2,t]}K_s(z,y) \leq  C t^{-n/2} \exp\(-|z-y|^2/ct\).$$
Therefore, for $0< t \leq 2$ and $0<\gamma<1$ close to 1 we get
\begin{align*}
F_5(x,y;t)&\leq C \int_0^{t/2} \int_\Rn  t^{-\gamma}P_s(x-z)V_1(z) t^{-n/2+\gamma} \exp\(-|z-y|^2/ct\)\, dz\, ds\cr
&\leq C \int_\Rn |x-z|^{2-n} t^{-\gamma} \exp\(-|x-z|^2/ct\) V_1(z) |z-y|^{-n+2\gamma}\, dz\cr
&\leq C \int_\Rn |x-z|^{2-n-2\gamma} \exp\(-|x-z|^2/c'\) V_1(z) |z-y|^{-n+2\gamma}\, dz.
\end{align*}
Thus, by using Lemma \ref{Lr}, we get
\begin{align*}
\sup_{y\in \Rn}\int_\Rn \sup_{0<t\leq 2} F_5(x,y;t)dx \leq C.
\end{align*}

 To deal with $F_6$ we observe that for $0< t \leq 2$ and $0<\gamma<1$ close to 1 we have
\begin{align*}
F_6(x,y;t)&\leq C \int_\Rn t P_t(x-z) V_1(z_1) t^{-n/2} \exp\(-|z-y|^2/ct\)\, dz\cr
&\leq \int_\Rn |x-z|^{2-n-2\gamma} \exp\(-|x-z|^2/c'\) V_1(z) |z-y|^{-n+2\gamma}\, dz
\end{align*}
and, consequently,
$$\sup_{y\in \Rn} \int_\Rn \sup_{t<2} F_6(x,y;t)\,dx \leq C.$$
\end{proof}

\begin{lema}\label{lem3}
The operator $f \mapsto \sup_{t>2} |Q_t f|$ is bounded on $L^1(\RR^n)$.
\end{lema}
\begin{proof}
Notice that for $\e>0$ and $t>2$ we have
\begin{align}
\int_t^\8  K_s(z,y)\, ds \leq C \int_t^\8 s^{-\e} s^{-n/2+\e}
\exp\(-\frac{|y-z|^2}{4s}\)\, ds \leq C t^{-\e} |y-z|^{2-n+2\e}.
\end{align}
It causes no loss of generality to consider only $Q_t^\jjjj(x,y)$. If $t>2$, then
$$0\leq Q_t^\jjjj(x,y)\leq C\int_\Rn P_t(x-z)V_1(z) t^{-\e} |y-z|^{2-n+2\e}\, dz.$$
Since
$\sup_{t>2} t^{-\e} P_t(x-z) \leq C (1+|x-z|)^{-n-2\e}$,
we find that
\begin{align}
\sup_{y\in\Rn}\int_\Rn \sup_{t>2} Q_t^\jjjj(x,y)\, dx &\leq C \sup_{y\in\Rn} \int_\Rn \int_\Rn (1+|x-z|)^{-n-2\e} V_1(z)  |y-z|^{2-n+2\e} \, dz\, dx\cr
&\leq C \sup_{y\in\Rn} \int_\Rn V_1(z) |y-z|^{2-n+2\e} \, dz \leq C.
\end{align}
The last inequality follows from Lemma \ref{Lr}.
\end{proof}

\begin{lema}\label{lem4}
The operator $f \mapsto \sup_{t\leq 2} |Q_t f|$ is bounded on $L^1(\RR^n)$.
\end{lema}
\begin{proof}
The estimate $\int_t^\8 K_s (z,y) \, ds \leq C |z-y|^{2-n}$ implies
$$\sup_{t\leq 2} Q_t(x,y) \leq C \sup_{t\leq 2} \int_\Rn P_t(x-z)V(z)|z-y|^{2-n}\, dz.$$
We claim that for fixed $y\in \Rn$ the foregoing function (of variable $x$) belongs to $L^1(\RR^n)$ and
$$\sup_{y\in \Rn} \int_\Rn \sup_{t\leq 2} Q_t(x,y)\, dx <\8.$$
The claim follows by arguments identical to that we used to prove \eqref{rtyui}.
\end{proof}

Now, Theorem \ref{tw1} follows directly from Lemmas \ref{lem1}, \ref{lem2}, \ref{lem3}, \ref{lem4}.

\section{Proof of Theorem \ref{tw2}}
\begin{proof}
Thanks to  (\ref{eqeq}) and Proposition \ref{lem_w}, for $g\in L^1(\mathbb R^n)$, we obtain
\begin{equation}\label{rel_ab}
\begin{split}
\int_{\Rn} (I-VL^{-1})(g\slash \omega)(x)\, dx & = \int_{\Rn}\frac{g(x)}{\w(x)}\, dx-\int_{\Rn}\int_{\Rn} V(x)\Gamma (x,y) \frac{g(y)}{\w(y)}\, dy\, dx\\
 & =\int_{\Rn}\frac{g(x)}{\w(x)}\, dx-\Big(\int_{\Rn}\frac{g(y)}{\w(y)}\, dy -w(y) \frac{g(y)}{\w(y)}\, dy\Big) \\
 &=\int_{\Rn} g(y)\, dy.
\end{split}\end{equation}

First, we are going to prove that
\begin{equation}\label{ineq1}
\|\w f\|_{H^1_\Delta(\RR^n)}\leq \|f\|_{H^1_L(\RR^n)}.
\end{equation}
Theorem \ref{tw1} combined with \eqref{odwrocenie} implies that \eqref{ineq1} is equivalent to
\begin{equation}\label{dfgh}
\|\w(I-V\Delta^{-1})f\|_{H^1_\Delta(\RR^n)} \leq C\|f\|_{H^1_\Delta (\RR^n)}.
\end{equation}

Assume that $a$ is a classical  $(1,\8)-$atom associated with $B=B(y_0,r)$, i.e.,
\begin{equation}\label{atomatom}
\supp\, a \subseteq B, \quad \|a\|_\8\leq |B|^{-1}, \quad \int_B a(x)\, dx = 0.
\end{equation}
By the atomic characterization of $H^1_\Delta(\RR^n)$ the inequality \eqref{dfgh} will be obtained, if we have established that $b = \w(I-V\Delta^{-1})a\in H^1_\Delta (\RR^n)$ and
\begin{equation}\label{ogratom}
\|b\|_{ H^1_\Delta (\RR^n)}\leq C
\end{equation}
with a constant $C>0$ independent of $a$.

By (\ref{odwrocenie}), $a=(I-VL^{-1})(b\slash \w)$. Hence, using (\ref{rel_ab}) we get
\begin{equation}\label{zero}
\int_\Rn b(x)\, dx =0.
\end{equation}

Proof of \eqref{ogratom} is divided into two cases.\\

{\it Case 1: $r\geq 1$.} Set
$$
b(x)=(b(x) - c_1)\chi_{2B}(x) + \sum_{k=2}^\8\(b(x) \chi_{2^kB\backslash 2^{k-1}B}(x)+c_{k-1}\chi_{2^{k-1}B}(x) - c_k\chi_{2^kB}(x)\)
=\SumNN b_k(x),
$$
where
$$c_k= -|2^kB|^{-1} \int_{(2^kB)^c}b(x)\, dx, \qquad  k=1,2,...\ .$$
Here and throughout,  $\rho B=B(y_0,\rho r)$ for $B=B(y_0,r)$.

We claim that
\begin{equation}\label{sumhar}
\SumNN \| b_k\|_{H^1_\Delta(\RR ^n)}\leq C.
\end{equation}

From Lemma \ref{Lrr} and Proposition \ref{lem_w} we conclude that there exists $\sigma >0$ such that
\begin{equation}\label{fdgg}
\begin{split}
|c_k|&\leq |2^kB|^{-1}\int_{(2^kB)^c}V(x)|\Delta^{-1}a(x)|\, dx \leq C |2^kB|^{-1}\int_{(2^kB)^c}\int_BV(x)\,|x-y|^{2-n}|a(y)|\, dy\, dx\\
&\leq C |2^kB|^{-1} \int_B |a(y)| \int_{(2^kB)^c}V(x)\,|x-y_0|^{2-n}\, dx \, dy \leq C|2^kB|^{-1} (2^kr)^{-\sigma}.
\end{split}
\end{equation}
Note that $\supp \, b_k \subseteq 2^kB$ and $\int_\Rn b_k(x)\, dx=0$. Therefore \eqref{sumhar} follows, if we have verified  that there exists $q>1$ such that
\begin{equation}\label{aaaa3}
\SumNN \|b_k\|_{L^q(\Rn)} |2^kB|^{1-1/q} \leq C,
\end{equation}
where $C$ does not depend on $a$.

If  $k=1$, then
$$|b_1(x)|\leq |c_1| \chi_{2B}(x) + |a(x)| +  V(x) |\Delta^{-1}a(x) |\chi_{2B}(x)$$ and
$$\|b_1\|_{L^q(\Rn)}\leq C|2B|^{-1+1/q} + \(\int_{2B} V(x)^q |\Delta^{-1}a(x)|^q dx\)^{1/q}.$$
Notice that
$$\(\int_{2B} V(x)^q |\Delta^{-1}a(x)|^q dx\)^{1/q}\leq C r^2|B|^{-1} \Summ \(\int_{2B} V_j(x)^q dx\)^{1/q}.$$
We can consider only the summand with $V_1$. By the H\"older inequality,
\begin{align*}
r^2|B|^{-1} \(\int_{2B} V_1(x)^q dx\)^{1/q}&\leq C r^2|B|^{-1} r^{\wt n_1/q} \|V_1\|_{L^{qs}(\VVV)} r^{n_1(1-1/s)/q}\cr
&=C|B|^{-1+1/q} r^{2-n_1/(sq)}.
\end{align*}
Choosing $q,s>1$ such that $V_1\in L^{qs}(\VVV)$ and $2-n_1/(qs)<0$ we get
\begin{equation}\label{aaaa1}
\|b_1\|_{L^q(\Rn)}\leq C |2B|^{-1+1/q}.
\end{equation}

For $k>1$, by the definition of $b_k$, we get that
$$\|b_k\|_{L^q(\Rn)} \leq |c_{k-1}||2^{k-1}B|^{1/q} +|c_{k}||2^{k}B|^{1/q} + \|b\|_{L^q(2^kB \backslash 2^{k-1}B)}$$
From \eqref{fdgg} we see that first two summands can be estimated by $C|2^kB|^{-1+1/q}2^{-k\sigma}.$ Then it remains to deal with the last summand. By using Lemma \ref{Lrr} there exists $\sigma ' >0$ such that for $q\in (1,1+\e]$ we have
\begin{equation}\label{aaaa2}
\begin{split}
\|b\|_{L^q(2^kB \backslash 2^{k-1}B)} &\leq C\(\int_{2^kB\backslash 2^{k-1}B}\( \int_B V(x)|x-y|^{2-n} |a(y)|\, dy\)^q dx\)^{1/q}\\
&\leq C \(\int_{(2^{k-1}B)^c} V(x)^q |x-y_0|^{q(2-n)}dx\)^{1/q}\leq C (2^kr)^{-\sigma '}\\
&= C|2^kB|^{-1+1/q}(2^kr)^{-\sigma '+n-n/q} \leq C|2^kB|^{-1+1/q}2^{-k\delta}
\end{split}
\end{equation}
provided that  $\delta =-\sigma '+n-n/q>0$.

The estimate \eqref{aaaa3} follows from \eqref{aaaa1} and \eqref{aaaa2}. This ends Case 1.

{\it Case 2: $r<1$.} Fix $N\in \NN\cup \{0\}$ such that $1/2 <2^Nr \leq 1$. Then
\begin{align*}
b(x)=&(a(x)\w(x)-c_0\chi_B(x))+\sum_{i=1}^N c_0|B| \(|2^{i-1}B|^{-1}\chi_{2^{i-1}B}(x)-|2^{i}B|^{-1}\chi_{2^{i}B}(x)\)\cr
&+\(b(x)-a(x)\w(x) +c_0 |B||2^NB|^{-1}\chi_{2^NB}(x)\) = d_0(x)+\sum_{i=1}^N d_i(x) + b'(x),
\end{align*}
where
$$c_0 =|B|^{-1}\int_B a(x)\w(x)\, dx.$$

By using $\int_B a =0$ and property \eqref{prop_a} from Proposition \ref{lem_w}, we obtain
\begin{equation}\label{c00}
|c_0|\leq |B|^{-1}\int_B |a(x)||\w(x)-\w(y_0)|dx\leq r^\delta |B|^{-1}.
\end{equation}

Observe that $\supp\, d_0 \subseteq B$, $\int_B d_0 =0$, and $\|d_0\|_\8 \leq C|B|^{-1}$. Similarly, for $i=1,...,N$, $\supp\, d_i\subseteq 2^iB$, $\int d_i =0$ and $\|d_i\|_\8\leq Cr^\delta |2^iB|^{-1}$.
Therefore
\begin{align*}
\sum_{i=0}^N\|d_i\|_{\HarD} \leq C+CNr^{\delta} \leq C-Cr^\delta \log_2 r\leq C.
\end{align*}

Denote $B' = 2^NB$. Obviously $|B'|\sim 1$. To deal with $b'(x)$ we apply the method from Case~1 with respect to $B'$, i.e.,
$$
b'=(b'(x) - c'_1)\chi_{2B'}(x) + \sum_{k=2}^\8\(b'(x) \chi_{2^kB'\backslash 2^{k-1}B'}(x)+c'_{k-1}\chi_{2^{k-1}B'}(x) - c'_k\chi_{2^kB'}(x)\)
=\SumNN b'_k,
$$
where
$$c'_k=
 -|2^kB'|^{-1} \int_{(2^kB')^c}b'(x)\, dx.$$
The arguments that we used in Case 1 also give
\begin{equation}\label{c000}
|c_k'|\leq C|2^kB'|^{-1} 2^{-k\sigma}\ \ \text{for } k=1,2,... \quad \text{and} \quad \sum_{k=2}^\8 \| b'_k\|_{H^1_\Delta(\RR ^n)}\leq C.
\end{equation}
It remains to obtain that
\begin{equation}\label{vbn}
\| b'_1\|_{H^1_\Delta(\RR ^n)}\leq C.
\end{equation}
It is immediate that $\supp\, b_1' \subseteq 2B'$ and $\int_{2B'}b_1' =0$. Also,
\begin{equation}\label{ooop}
\|b_1'\|_{L^q(\Rn)} \leq \(\int_{2B'}V(x)^q|\Delta^{-1}a(x)|^q\)^{1/q} + C|c_0||B||2B'|^{-1+1/q}+C|c_1'||2B'|^{1/q}.
\end{equation}
By \eqref{c00} and \eqref{c000} only the first summand needs to be estimated. Observe that
$$|\Delta^{-1}a(x)| \leq \int_B|x-y|^{2-n}|a(y)|\, dy \leq
\left\{\begin{array}{l l}
Cr^{2-n} & \text{if } |x-y_0|<2r\\
C|x-y_0|^{2-n} & \text{if }|x-y_0|>2r
\end{array}
\right\}
\leq C|x-y_0|^{2-n}.$$
Therefore, by using Lemma \ref{Lr}, we get
\begin{equation*}
\|b_1'\|_{L^q(\Rn)}\leq C
\end{equation*}
and \eqref{vbn} follows, which finishes Case 2 and the proof of \eqref{ineq1}.

In order to finish the proof of Theorem \ref{tw2} it remains to prove that
\begin{equation}\label{ppppp}
\|f\|_{H^1_L(\Rn)} \leq C \|\w f\|_{H^1_\Delta(\RR^n)}.
\end{equation}
In virtue of Theorem \ref{tw1} the inequality \eqref{ppppp} is equivalent to
\begin{equation}\label{dfghdfgh}
\big\|(I-VL^{-1})\(g/\w\)\big\|_{H^1_\Delta(\RR^n)} \leq C\|g\|_{H^1_\Delta (\RR^n)}.
\end{equation}

Assume that $a$ is an $H^1_\Delta (\Rn)$-atom (see \eqref{atomatom}). Set $b=(I-VL^{-1})(a/\w)$. The proof will be finished if we have obtained that
\begin{equation}\label{ogratom2}
\|b\|_{H^1_\Delta(\Rn)}\leq C
\end{equation}
with $C$ independent of atom $a$. By \eqref{rel_ab}, we have
$$\int_\Rn b(x)\,dx =\int_\Rn a(x)\, dx=0.$$
Note that  the proof of \eqref{ogratom} only relies on  estimates of  $|\Gamma_0(x,y)|$ from above by $C|x-y|^{2-n}$. The same estimates hold for $|\Gamma (x,y)|$. Moreover, the weight $1/\w$ has the same properties as $\w$, that is,  boundedness from above and below by positive constants and the H\"older condition. Therefore the proof of \eqref{ogratom2} follows by the same arguments. Details are omitted.
\end{proof}

\section{Proof of Theorem \ref{tw3}}
By \eqref{pert} we get a formula similar to \eqref{pert1}.
\begin{equation}\label{pert3}
K_t- P_t(I-VL^{-1}) = Q_t'-W_t',
\end{equation}
where
\begin{equation*}\label{pert4}
W_t' = \int_0^t P_{t-s}\,V\, K_s\, ds, \qquad Q_t' = \int_0^\8 P_t \, V \, K_s \, ds.
\end{equation*}

Recall that for $i=1,...,n$ we denote by $\dxi$ the derivative in the direction of $i$-th standard coordinate. For $f \in \LL$ from \eqref{pert1} and \eqref{pert3} we get
\begin{equation}\label{rrrr}
\iie \dxi K_t f \dtt - \iie \dxi  P_t (I-VL^{-1}) f \dtt = \ww_\ie 'f + \qq_\ie 'f +\ww_\ie f +\qq_\ie f,
\end{equation}
\begin{align*}
\qq_\ie = \int_2^{\e^{-1}} \dxi Q_t \dtt, \qquad \qq_\ie ' = \int_\e^2 \dxi Q_t'\dtt,\cr
\ww_\ie = -\int_2^{\e^{-1}} \dxi W_t \dtt, \qquad \ww_\ie ' = -\int_\e^2 \dxi W_t' \dtt.
\end{align*}

All the operators above are well-defined and bounded on $\LL$. By the theory of the classical Hardy spaces $\rr_{\Delta,i} f= \lim_{\e \to 0} \iie \dxi P_t f \dtt \in \LL$ for every $i=1,...,n$, exactly when $f\in \HarD$. Moreover,
\begin{equation}\label{jjj}
\|f\|_{\HarD} \sim \|f\|_{\LL} + \sum_{i=1}^n \|\rr_{\Delta,i} f\|_{\LL}.
\end{equation}

The subsequent four lemmas prove that the operators $\qq_\ie, \qq_\ie ',\ww_\ie, \ww_\ie '$ converge strongly as $\e \to 0$ in the space of $\LL$-bounded operators.

\begin{lema}
For every $i=1,...,n$ the operators $\qq_\ie$ converge as $\e\to 0$ is norm-operator topology.
\end{lema}
 \begin{proof}
The operators $\qq_\ie$ have the integral kernels
$$\qq_\ie(x,y)=\int_2^{\e^{-1}} \int_t^\8 \int_\Rn \dxi P_t(x-z)V(z)K_s(z,y)\, dz \, ds \, \dtt.$$
The lemma will be proved when we have obtained
$$\sup_{y\in \Rn} \int_\Rn \QQ_i^{(j)}(x,y) dx \leq C,$$
where
$$
\QQ_i^{(j)}(x,y)=\int_2^\8 \int_t^\8 \int_\Rn |\dxi P_t(x-z)|V_j(z)K_s(z,y)\, dz \, ds \, \dtt.
$$
Since $|\dxi P_t(x-z)|\leq Ct^{-1/2} \phi_t(x-z)$ for some $\phi \in \ss(\RR^n)$ we get
 \begin{equation}
\begin{split}
\int_\Rn \QQ_i^{(j)}(x,y) dx &\leq C \int_\Rn \int_2^\8 \int_	t^\8 \int_\Rn t^{-1/2} \phi_t(x-z) V_j(z)K_s(z,y)\, dz \, ds \, \dtt\, dx\cr
 &\leq C \int_2^\8 \int_t^\8 \int_\Rn t^{-1} V_j(z)K_s(z,y)\, dz \, ds \, dt\cr
 &\leq C \int_2^\8 \int_t^\8 \int_\Rn t^{-1-\e} V_j(z) s^{-n/2+\e}\exp(-|z-y|^2/4s)\, dz \, ds \, dt\cr
 &\leq C \(\int_2^\8 t^{-1-\e}dt\) \cdot\( \int_\Rn V_j(z)|z-y|^{2-n+2\e}\, dz\)\leq C,
\end{split}
 \end{equation}
where in the last inequality we used Lemma \ref{Lr} and $C$ does not depend on $y\in \Rn$.
 \end{proof}

\begin{lema}
For every $i=1,...,n$ the operators $\ww_\ie$ converge as $\e \to 0$ in norm-operator topology.
\end{lema}
 \begin{proof}
 The operators $\ww_\ie$ have the integral kernels
\begin{align*}
\ww_\ie(x,y)&=\int_2^{\e^{-1}} \int_0^t \int_\Rn \dxi \( P_{t-s}(x-z)-P_t(x-z)\)V(z)K_s(z,y)\, dz \, ds \, \dtt.
\end{align*}
Set
$$\WW_i^{(j)}(x,y) = \int_2^{\8} \int_0^t \int_\Rn \big|\dxi \( P_{t-s}(x-z)-P_t(x-z)\)\big|V_j(z)K_s(z,y)\, dz \, ds \, \dtt.$$
The lemma will be proved when we have obtained that
\begin{equation}\label{rrreee}
\sup_{y\in \Rn} \int_\Rn \WW_i^{(j)}(x,y) dx \leq C.
\end{equation}

For fixed $y\in \Rn$ and $0<\beta<1$, $\beta$ will be determined later on, we write
\begin{align*}
 \int_\Rn \WW_i^{(j)}(x,y) dx &\leq \int_\Rn \int_2^\8 \int_0^t \int_\Rn \big|\dxi\( P_{t-s}(x-z)-P_t(x-z)\)\big|V_j(z)K_s(z,y)\, dz \, ds \, \dtt\, dx\cr
 &\leq \int_0^{t^\beta}...\,ds + \int_{t^\beta}^t...\, ds = J_1+J_2.
\end{align*}

 Observe that there exists $\psi \in \ss(\RR^n)$, $\psi\geq 0$ such that for $s\in(0,t^{\beta})$ and $t>2$ we have
 $$\Big|\dxi\( P_{t-s}(x)-P_t(x)\)\Big|\leq st^{-3/2} \psi_t(x).$$
 Thus by using Lemma \ref{Lr} we get
\begin{equation}
\begin{split}\label{bbbhh}
 J_1 &\leq \int_\Rn \int_2^\8 \int_0^{t^{\beta}} \int_\Rn st^{-2}\psi_t(x-z)V_j(z)K_s(z,y)\, dz \, ds \, dt\, dx\cr
 &\leq C \int_2^\8 t^{-2+\beta} \, dt \cdot \int_\Rn V_j(z) |z-y|^{2-n}dz\leq C_1.
\end{split}
\end{equation}
 Note that if $t>2$ and $s\in [t^\beta,t]$ then $K_s(z)\leq Ct^{-\beta n/2} \exp(-|z|^2/ct)$. Choosing $0<\beta<1$, $\beta $ close to 1, and applying Lemma \ref{Lr} we obtain
\begin{equation}\label{bbbh}
\begin{split}
J_2 &\leq \int_\Rn \int_2^\8 \int_{t^\beta}^t \int_\Rn \(\frac{\psi_{t-s}(x-z)}{\sqrt{t-s}} +\frac{\psi_t(x-z)}{\sqrt{t}}\)V_j(z)K_s(z,y)\, dz \, ds \, \dtt\, dx\cr
&\leq C\int_2^\8 \int_0^t \int_\Rn \(((t-s)t)^{-1/2} +t^{-1}\)V_j(z)t^{-\beta n/2} \exp(-|z-y|^2/ct)\, dz \, ds \, dt\cr
&\leq C \int_2^\8 \int_\Rn V_j(z)t^{-\beta n/2} \exp(-|z-y|^2/ct)\, dz \, dt
\leq C\int_\Rn V_j(z) |z-y|^{2-\beta n}dz \leq C_2.
\end{split}
\end{equation}
 Notice that the  constants $C_1$ and $C_2$  in \eqref{bbbhh} and \eqref{bbbh} respectively do not depend on $y\in\Rn$. Thus \eqref{rrreee} follows.
 \end{proof}

\begin{lema}
For $i=1,...,n$ the operators $\ww_\ie '$ converge as $\e \to 0$ in norm-operator topology.
\end{lema}
\begin{proof}
The operators $\ww_\ie '$ have the integral kernel
$$\ww_\ie '(x,y)=\int_\e^2 \int_0^t \int_\Rn \dxi P_{t-s}(x-z)V(z)K_s(z,y)\, dz \, ds \, \dtt.$$
The lemma will be proved if we have shown that
\begin{equation}\label{nnn}
\sup_{y\in \Rn}\int_\Rn \WW_i^{(j)'}(x,y)\, dx \leq C,
\end{equation}
where
$$ \WW_i^{(j)'}(x,y) =  \int_0^2 \int_0^t \int_\Rn \big| \dxi P_{t-s}(x-z)\big|V_j(z)K_s(z,y)\, dz \, ds \, \dtt .$$

Fix $y\in \Rn$. Observe that
\begin{align*}
\int_\Rn \WW_i^{(j)'}(x,y) dx &\leq \int_\Rn \int_0^2 \int_0^t \int_\Rn \big|\dxi P_{t-s}(x-z)\big|V_j(z)K_s(z,y)\, dz \, ds \, \dtt\, dx\cr
&\leq \int_0^{t/2}...\, ds+\int_{t/2}^t...\, ds = J_3+J_4.
\end{align*}
There exists $\psi \in \ss(\RR^n)$, $\psi \geq 0$, such that
\begin{align*}
J_3&\leq  \int_\Rn \int_0^2 \int_0^{t/2} \int_\Rn (t(t-s))^{-1/2}\psi_{t-s}(x-z)V_j(z)K_s(z,y)\, dz \, ds \, dt \, dx\cr
&\leq C \int_0^2 \int_0^t \int_\Rn t^{-1}V_j(z)K_s(z,y)\, dz \, ds \, dt\cr
&\leq C \int_0^2 \int_\Rn t^{-1}V_j(z)|z-y|^{2-n} \exp\(-|z-y|^2/ct\) dz \, dt\cr
&\leq C \int_{|z-y|> 1/2} V_j(z)|z-y|^{2-n} dz + \int_{|z-y|\leq 1/2} V_j(z)|z-y|^{2-n}|\log|z-y|| dz \leq C_3
\end{align*}
and
\begin{align*}
J_4&\leq C \int_\Rn \int_0^2 \int_{t/2}^t \int_\Rn (t(t-s))^{-1/2}\psi_{t-s}(x-z)V_j(z)t^{-n/2}\exp\(-\frac{|z-y|^2}{ct}\)\, dz \, ds \, dt \, dx\cr
&\leq C \int_0^2 \int_0^{t/2} \int_\Rn (ts)^{-1/2}V_j(z)t^{-n/2}\exp\(-\frac{|z-y|^2}{ct}\)\, dz \, ds \, dt\cr
&\leq C \int_\Rn V_j(z) \int_0^\8 t^{-n/2}\exp\(-\frac{|z-y|^2}{ct}\)\, dt \, dz\leq C \int_\Rn V_j(z) |z-y|^{2-n}dz\leq C_4
\end{align*}
with constants $C_3$ and $C_4$ independent of $y\in \Rn$. So we have obtained \eqref{nnn}.
\end{proof}

\begin{lema}
For $i=1,...,n$ the operators $\qq_\ie '$ converge strongly as $\e\to 0$.
\end{lema}
\begin{proof}
The kernels of $\qq_\ie '$ are given by
$$\qq_\ie '(x,y) = \int_\e^2\int_0^\8 \int_\Rn \dxi P_t(x-z) V(z) K_s(z,y)\, dz \, ds \, \dtt.$$
For $f\in L^1(\RR^n)$ we have
$$\qq_\ie 'f(x) = \int_\Rn \qq_\ie '(x,y) f(y)\, dy.$$
Note that $\qq_\ie '(x,y)= H_\ie * \phi_y(x)$, where $\phi_y (z) =V(z) \Gamma(z,y)$ and $H_\ie (x)=\int_\e^2 \dxi P_t (x)\dtt$.

It follows from the theory of singular integrals operators that for $g\in L^r(\RR^n)$, $r>1$, the limits $\lim_{\e \to 0}  H_\ie * g(x) = H_i g(x)$ exist for a.e. $x$ and in $L^r(\Rn)$ norm. Obviously, $H_i$ are $L^r(\Rn)$-bounded operators. Moreover,
\begin{align}\label{rtyu}
\Big\|\sup_{0<\e<2} |H_\ie * g|\Big\|_{L^r(\Rn)} \leq C \|g\|_{L^r(\Rn)}.
\end{align}

Notice that for $|z|>1/2$ w have
\begin{equation}\label{ppss}
\sup_{0<\e<2}|H_\ie(z)| \leq C_N|z|^{-N}.
\end{equation}
From \eqref{rtyu} and  \eqref{ppss} we deduce that if $a$ is a function supported in a ball $B(y_0, R)$, $R>1/2$, and $\|a\|_{L^r(\Rn)}\leq \tau |B|^{-1+1/r}$, $r>1$, then
\begin{equation}\label{sdfg}
\Big\|\sup_{0<\e<2} |H_\ie* a|\Big\|_{L^1(\RR^n)}\leq C \tau.
\end{equation}

Using Lemma \ref{Lr} we get that for every $y\in \Rn$ the limit $\lim_{\e\to 0}Q_\ie '(x,y)=Q_i '(x,y)$ exists for a.e. $x\in \Rn$.
The lemma will be proved by using the Lebesque's dominated convergence theorem if we have established that:
\begin{equation}\label{ppa1}
\sup_{y\in \Rn} \int_\Rn \sup_{0<\e<2}|\qq_\ie '(x,y)|\, dx \leq C \quad \text{and}
\end{equation}
\begin{equation}\label{ppa2}
\lim_{\e \to 0} \int_\Rn |\qq_\ie '(x,y)-\qq_i '(x,y)|\, dx =0 \qquad \text{for every }y.
\end{equation}

For fixed $y\in \Rn$ let
$$\phi_1(z)=\phi_y(z)\chi_{B(y,2)}(z), \qquad \phi_k(z)=\phi_y(z)\chi_{B(y,2^k)\backslash B(y,2^{k-1})}(z), \ k\geq 2.$$
Then $\phi_y = \sum_{k=1}^\8 \phi_k$, where the series converges in $L^1(\Rn)$ and $L^r(\RR^n)$ norm for $r$ slightly bigger than 1. Notice that $\supp \, \phi_k \subseteq B(y,2^k)$, $\|\phi_1\|_{L^r(\Rn)}\leq C$, and
\begin{align}
\|\phi_k\|_{L^r(\Rn)}^r &= \int_{B(y,2^k)\backslash B(y,2^{k-1})}V_1(z)^r|z-y|^{(2-n)r}\, dz\leq 2^{k(2-n)r} \int_{B(y,2^k)}V_1(z)^r\, dz\cr
&\leq C 2^{k(2-n)r} 2^{k(n-n_1)} \|V_1\|^r_{L^{rq}(\VVV)} 2^{kn_1/q'}=C(2^k)^{-nr+n+2r-n_1/q}.
\end{align}
Therefore, for $q<n_1/2r$ such that $V_1\in L^{rq}(\VVV)$, we get
\begin{equation}\label{edc}
\|\phi_k\|_{L^r(\Rn)}\leq C |B(y,2^k)|^{-1+1/r} 2^{-\sigma k},
\end{equation}
where $\sigma = n_1/(qr) - 2 >0$.
By using \eqref{sdfg} combined with \eqref{edc} we obtain
\begin{align}
\int_\Rn \sup_{0<\e<2}|\qq_\ie '(x,y)|\, dx &= \int_\Rn \sup_{0<\e<2}|H_\ie \phi_y(x)|\, dx \cr
&\leq \sum_{k=1}^\8 \int_\Rn \sup_{0<\e<2}|H_\ie\phi_k(x)|\, dx \cr
&\leq C \sum_{k=1}^\8 2^{-\sigma k} \leq C,
\end{align}
which implies \eqref{ppa1}, since the last constant $C$ does not depend of $y$. Additionally \eqref{ppa2} is a consequence of \eqref{ppa1} and Lebesque's dominated convergence theorem.
\end{proof}

Now, Theorem \ref{tw3} follows directly by applying \eqref{rrrr}, \eqref{jjj}, and Theorem \ref{tw1}.


\end{document}